\documentclass[12pt]{article}
\usepackage{mathptmx,amsfonts,amssymb}

\newtheorem{theorem}{Theorem}
\newtheorem{lemma}[theorem]{Lemma}
\newtheorem{corollary}[theorem]{Corollary}
\newtheorem{unnumber}{}

\newtheorem{prepf}[unnumber]{Proof}
\newenvironment{proof}{\prepf\rm}{\endprepf}
\newcommand{\qed}{\hfill$\Box$}
\newtheorem{prerk}[unnumber]{Remark}
\newenvironment{remark}{\prerk\rm}{\endprerk}
\newtheorem{preex}[unnumber]{Example}
\newenvironment{example}{\preex\rm}{\endpreex}

\begin{document}
\title{On the connectivity and independence number of power graphs of groups}
\author{P. J. Cameron\\
School of Mathematics and Statistics, University of St~Andrews\\
North Haugh, St Andrews, Fife KY16~9SS, UK\\
\texttt{pjc20@st-andrews.ac.uk}\\
and\\
S. H. Jafari\\
Faculty of Mathematical Sciences, Shahrood University of Technology\\
P. O. Box 3619995161-316, Shahrood, Iran\\
\texttt{shjafari55@gmail.com}}
\date{}

\maketitle

\begin{abstract}
Let $G$ be a  group. The power graph  of $G$ is a graph with vertex set $G$
in which two distinct elements $x,y$ are adjacent if one of them is a power
of the other.  We characterize all  groups whose power graphs have finite
independence number, show that they have clique cover number equal to
their independence number, and calculate this number.

The proper power graph is the induced subgraph of the power graph on the
set $G-\{1\}$. A group whose proper power graph is connected must be either
a torsion group or a torsion-free group; we give characterizations of some
groups whose proper power graphs are connected.



\end{abstract}

\section{Introduction}
We begin with some standard definitions from graph theory and group theory.

Let $G$ be a graph with vertex set $V (G)$. An \emph{independent set} is a set of vertices
in a graph, no two of which are adjacent; that is, a set whose induced
subgraph is null. The \emph{independence number} of a graph $G$ is the cardinality of the largest
independent set and is denoted by $\alpha(G)$. The \emph{chromatic number} of
$G$ is the minimum number of parts in a partition of $V(G)$ into independent
sets. Dually, a \emph{clique} is a set of vertices with all pairs adjacent; the
\emph{clique cover number} is the minimum number of parts in a partition of
$V(G)$ into cliques. Clearly we have
\begin{enumerate}\itemsep0pt
\item the clique number and chromatic number of a graph are equal to the
independence number and clique cover number of the complementary graph;
\item for any graph, the clique number is at most the chromatic number, and
the independence number is at most the clique cover number.
\end{enumerate}

The cyclic group of order $n$ is denoted by $C_n$. A group $G$ is called \emph{periodic} if every
element of $G$ has finite order. For every element $g \in G$, the order of $g$ is denoted by
$o(g)$. If there exists an integer $n$ such that for all $g \in G$, $g^n = 1$, where $1$ is
 the identity element of $G$, then $G$ is said to be of \emph{bounded exponent} and the \emph{exponent} of $G$ is
  $\exp(G)=\min \{n\mid g^n=1 \mbox{ for } g \in G\}$.

 A group $G$ is said to be \emph{torsion-free} if apart from the identity every element of $G$ has infinite
  order.

Let $p$ be a prime number.
The $p$-\emph{quasicyclic} group (known also as the \emph{Pr\"{u}fer group}) is the $p$-primary component
of $ \mathbb{Q}/\mathbb{Z}$. It is denoted by $C_{p^{\infty}}$.

The center of a group $G$, denoted by $Z(G)$, is the set of elements that commute with every
element of $G$. A group $G$ is called \emph{locally finite} if every finitely generated subgroup of
$G$ is finite. Obviously any  locally finite group is periodic. We say a group $G$ is \emph{locally center-by-finite}
 if for any finite subset $X$ of $G$, $[\langle X\rangle: Z(\langle X\rangle)]< \infty$.

A module $M$ is an \emph{essential extension} of a
 submodule $N$, if  any nonzero  submodule of $M$ has nonzero intersection with $N$. A module $M$ is a maximal essential extension of N if $K$ is an arbitrary 
  essential extension of $N$ and $M \subseteq K$, then $K=M$.

 The \emph{directed power graph} of $G$ is the directed graph $\vec{P}(G)$ which
takes $G$ as its vertex set with an edge from $x$ to $y$ if $y\ne x$ and
$y$ is a positive power of $x$. The (undirected) \emph{power graph}
of $G$, denoted $P(G)$,  takes $G$
as its vertex set with an edge between distinct elements if one is a positive power
of the other. For example, if $G=C_{p^\infty}$, then $P(G)$ is a countably
infinite complete graph. The \emph{proper power graph} of $G$ is
$P^*(G)=P(G)-{1}$.

Directed power graphs were first define by Kelarev
and Quinn \cite{r4} to study semigroups, and undirected power graphs of groups were introduced by
Chakrabarty \textit{et al.}~\cite{r2}. The reference \cite{r1} surveys a number of results on power
graphs. In \cite{ali},  Aalipour \emph{et al.}~provide some results on the finiteness of the independence number of
power graphs. They proved, if $G$ is an infinite nilpotent group, then
$\alpha(P(G))=\alpha(P^*(G)) <\infty $ if and only if
$G \cong C_{p^{\infty}} \times H$, where $H$ is a finite group and $p\nmid |H|$. They posed  the question, does this hold without assuming nilpotence?

In Section 2 of this paper we give an affirmative answer to this question,
and compute the independence number and clique cover number of the power
graphs of infinite groups of this form (they must be equal).

In Section 3, we consider the question of connectivity of proper power graphs
of infinite groups.

\section{Independence number}

In this section we give a description of groups whose power graph has finite
independence number, and show that, for such groups, the independence number
and clique cover number of the power graph are equal.

We need the following theorems \cite[Theorems 1 and 2]{ali}.

\begin{theorem}\label{ali1}
Let $G$ be a group satisfying $\alpha(P(G)) < \infty$. Then
\begin{enumerate}\itemsep0pt
\item $[G : Z(G)] < \infty$.
\item $G$ is locally finite.
\end{enumerate}
\end{theorem}

\begin{theorem}\label{ali2}  Let $G$ be an abelian group satisfying
 $\alpha(P(G)) <\infty$. Then either $G$ is finite
or $G \cong C_{p^{\infty}} \times H$, where $H$ is a finite group and $p\nmid |H|$.
\end{theorem}

Now we settle an open problem of Aalipour \emph{et al.}
\cite[Question 38]{ali}.

\begin{theorem} Let $G$ be a group satisfying $\alpha(P(G)) <\infty$. Then either $G$
 is finite,
or $G \cong C_{p^{\infty}} \times H$, where $H$ is a finite group and $p\nmid |H|$.
\label{t:fin_ind}
\end{theorem}

\begin{proof}
First suppose that $G \cong C_{p^{\infty}} \times  H$, where $H$ is a finite group and
 $p\nmid |H|$. Assume that
$S$ be a infinite independent set of $P(G)$. Since $H$ is finite, $S$ has an infinite subset
 $S_1=\{(a, h)\mid a \in K $ for some $h\in H$ and $K \subseteq C_{p^{\infty}} $. 
 Let $(a,h), (b,h) \in S_1$ and $\langle a\rangle \subseteq \langle b \rangle$. Since $\langle (a,h) \rangle = \langle a\rangle \times \langle h\rangle$ then
  $\langle (a,h) \rangle  \subseteq  \langle (b,h) \rangle$.  
Thus  all vertices in $S_1$ are adjacent, a contradiction.

Conversely, suppose that $\alpha(P(G)) < \infty$ and $G$ is infinite. Then by Theorem \ref{ali1},
 $[G : Z(G)] <\infty$ and
so $G = Z(G)H_1$, where $H$ is a finitely generated subgroup of $G$. Theorem \ref{ali1} implies
that $H_1$ is finite. By Theorem \ref{ali2}, $Z(G) = AB$, where $A \cong C_{p^{\infty}}$ and $B$
 is a finite group
such that $p \nmid |B|$. Consequently, $G=ABH_1=AH_2$, where $H_2=BH_1$. Since
$B$ is a central subgroup of $G$, then $H_2$ is a subgroup of $G$. 

Note that $A$ is a normal subgroup of finite index in $G$ . If $p\mid|G/A|$, 
then $G$ has an element $a$ of order $p$ outside $A$, and $A\langle a \rangle$
is an abelian $p$-group. Then  $A\langle a \rangle \cong C_{p^{\infty}} \times C_p$, which contradicts Theorem~\ref{ali2}.

Finally, the Sylow $p$-subgroup of $H_1$ is central, so by Burnside's transfer
theorem, $H_1$ has a normal $p$-complement, which is the required $H$. \qed
\end{proof}

As a corollary of this result, we show the following:

\begin{corollary}
Let $G$ be a group whose power graph $P(G)$ has finite independence number.
Then the independence number and clique cover number of $P(G)$ are equal.
\end{corollary}

\begin{remark}
This theorem can be deduced using \cite[Theorem 12]{ali}, asserting that the
power graph of a group of finite exponent is perfect, together with the
Weak Perfect Graph Theorem of Lov\'asz~\cite{lov}, asserting that the
complement of a finite perfect graph is perfect; this argument would also
require a compactness argument to show that the clique cover number of $P(G)$
is equal to the maximum clique cover number of its finite subgroups. However,
given our Theorem~\ref{t:fin_ind}, the argument below is substantially more
elementary.
\end{remark}

\begin{proof}
We know that $G=C_{p^\infty}\times H$, where $H$ is a finite group and $p$ a
prime not dividing $|H|$. We examine the structure of the power graph of such
a group. Let $\vec{P}(H)$ be the directed power graph of $H$; write
$h\leftrightarrow h'$ if $h\to h'$ and $h'\to h$ in $\vec{P}(H)$. Let $Z_i$
be the set of elements of order $p^i$ in $C_{p^\infty}$.

Note that $(z',h')$ is a power of $(z,h)$ if and only if $z'$ is a power of $z$
and $h'$ a power of $h$. One way round is trivial. In the other direction, 
suppose that $z'=z^a$ and $h'=h^b$. By the Chinese remainder theorem, choose
$c$ such that
\[c\equiv a\mbox{ mod }o(z),\quad c\equiv b\mbox{ mod }\exp(H).\]
Then $(z,h)^c=(z^a,h^b)=(z',h')$.

It follows that two elements $(z,h)$ and $(z',h')$ are adjacent in $P(G)$ if
and only if one of the following happens:
\begin{enumerate}\itemsep0pt
\item $h\leftrightarrow h'$ in $\vec{P}(G)$;
\item $h\to h'$ and $h'\not\to h$ in $\vec{P}(G)$, and $z\in Z_i$, $z'\in Z_j$
with $i\ge j$;
\item $h\not\to h'$ and $h'\to h$ in $\vec{P}(G)$, and $z\in Z_i$, $z'\in Z_j$
with $i\le j$.
\end{enumerate}
Now the relation $\leftrightarrow$ is an equivalence relation on $H$. If $E$
is an equivalence class, then $C_{p^\infty}\times E$ is a clique in $P(G)$,
so we have a clique cover of size equal to the number of $\leftrightarrow$
classes. We have to find an independent set of the same size.

To do this, we note that the $\leftrightarrow$ classes are partially ordered
by $\to$. Extend this partial order to a total order $<$, and number the
classes $E_1,\ldots,E_r$ with $E_1<E_2<\cdots<E_r$. Now take $h_i\in E_i$
and $z_j\in Z_j$; the set
\[\{(h_i,z_{r-i}):1\le i\le r\}\]
is an independent set, since if $i<j$ then $h_j\not\to h_i$ and $r-i>r-j$. \qed
\end{proof}

\begin{remark}
This argument also gives us a formula for the independence number of
$P(C_{p^\infty}\times H)$, where $p\nmid|H|$. Since $x\leftrightarrow y$ if
and only if $\langle x\rangle=\langle y\rangle$, we see that
$\alpha(P(C_{p^\infty}\times H))$ is equal to the number of cyclic subgroups
of $H$.
\end{remark}

\begin{corollary}
Let $G$ be a group for which $P(G)$ has finite independence number. Then
$P(G)$ is a perfect graph.
\end{corollary}

\begin{proof}
This is true if $G$ is finite, by \cite[Theorem 12]{ali}, so suppose that
$G=C_{p^\infty}\times H$, where $P\nmid|H|$. Any finite induced subgraph of
$G$ is contained in $C_{p^n}\times H$, for some $n$, and so is perfect. This
gives the result, since an infinite graph is defined to be perfect if all its
finite induced subgraphs are perfect. \qed
\end{proof}

\section{Connectivity}

In a torsion-free group $G$, the identity is an isolated vertex of $P(G)$,
while in a torsion group, it is joined to every vertex. So for questions of
connectivity, we use the proper power graph $P^*(G)$ instead.

However, a group may have elements other than the identity which are joined
to all vertices in the power graph. Our first result explains when this can
happen in a finite group.

Then we provide some results on the connectivity of proper power graphs, and
extend the just-mentioned result to infinite groups. 

\subsection{Vertices joined to every vertex}

In this subsection we consider finite groups only. We classify those
groups in which some non-identity vertex is joined to all others, and decide
whether the graphs remain connected when all such vertices are deleted.

Note that this is similar to the usual convention in studying the
\emph{commuting graph} of a group, the graph in which group elements $x$ and $y$
are joined if and only if $xy=yx$. In this case, the set of vertices joined to
all others is precisely the \emph{center} of the group, and in studying the
connectedness of the commuting graph it is customary to delete the center:
see for example \cite{gp,mp}. (We remark that the power graph of $G$ is a
spanning subgraph of its commuting graph.)

The \emph{generalized quaternion} group $Q_{2^n}$ is defined by
\[Q_{2^n}=\langle a, b \mid a^{2^{n-1}}=b^2, bab^{-1}=a^{-1}, b^4=1 \rangle.\]

\begin{theorem}
Let $G$ be a finite group. Suppose that $x\in G$ has the property that for all
$y\in G$, either $x$ is a power of $y$ or \emph{vice versa}. Then one of the
following holds:
\begin{enumerate}\itemsep0pt
\item $x=1$;
\item $G$ is cyclic and $x$ is a generator;
\item $G$ is a cyclic $p$-group for some prime $p$ and $x$ is arbitrary;
\item $G$ is a generalized quaternion group and $x$ has order~$2$.
\end{enumerate}
\label{t:joined_all}
\end{theorem}

\begin{proof}
We observe first that the converse is true; each of the four cases listed
implies that $x$ satisfies the hypothesis.

Note that the condition is inductive; that is, if $H$ is a subgroup of $G$
with $x\in H$, then $x$ satisfies the same hypothesis in $H$. If no such
subgroup apart from $G$ exists, then $G$ is cyclic and $x$ is a generator.
So we may inductively suppose that the theorem is true for any group smaller
than $G$. We may clearly assume that $x\ne1$.

We observe that, since an element and any power commute, $x$ belongs to the
center $Z(G)$ of $G$. Moreover, if $G$ is abelian, then it is cyclic. For if
not, then for some prime $p$, $G$ contains elements of order $p$ neither of
which is a power of the other; then they cannot be both powers of $x$.

Suppose first that the order of $x$ is a power of a prime $p$. Let $z$ be a
power of $x$ which has order $p$. Then clearly $\langle z\rangle$ is the only
subgroup of order $p$ in $G$. If $G$ is not a $p$-group, it contains an
element $u$ of prime order $q\ne p$; then $zu$ has order $pq$, and cannot be
a power of $x$. If $G$ is a $p$-group, then a theorem of Burnside
\cite[Theorem 12.5.2]{hall} shows that $G$ is either cyclic or generalized
quaternion; in the latter case, $x$ has order $2$.

So we may suppose that the order of $x$ is not a prime power. If $x\in H$ and
$H<G$, then $H$ must be cyclic generated by $x$. So $\langle x\rangle$ is a
maximal subgroup of $G$. In particular, the center of $G$ is generated by $x$,
but $G$ itself is not cyclic. Now elements outside $\langle x\rangle$
are not powers of $x$, and $x$ cannot be a power of such an element (else
$G$ would be cyclic). \qed
\end{proof}

Now we consider the result of deleting all such vertices.
The power graph of a cyclic group of prime power order is complete, so nothing
remains; but these groups present no difficulty.

Suppose that $G$ is cyclic of non-prime-power order. If the order of $G$ is
the product of two primes $p$ and $q$, then removing the identity and the
generators leaves a disconnected graph consisting of complete graphs of sizes
$p-1$ and $q-1$ with no edges between them. But in any other case, the graph
is connected. For any element of $G$ has a power which has prime order, and if
$x$ and $y$ are elements of distinct prime orders $p$ and $q$ then $x$ and $y$
are both joined to $xy$.

Finally, if we remove the identity and the involution from the generalized
quaternion group of order $2^n$, we obtain a complete graph of cardinality
$2^{n-1}-2$ together with $2^{n-2}$ disjoint edges.

In the remaining case, we delete the identity and obtain the proper power
graph $P^*(G)$. So from now on we consider only this case. 

We extend Theorem~\ref{t:joined_all} to infinite groups at the end of the next
subsection.

\subsection{Connectivity of the proper power graph}

Since an element of finite order is not adjacent to any element
of infinite order, we have the following elementary result.

\begin{lemma}\label{con}
If $P^*(G)$ is connected then $G$ is torsion-free or periodic. \qed
\end{lemma}

For  characterizing some torsion-free  groups  with connected proper
 power graphs we need the following famous theorems.
 
\begin{theorem}[Schur's theorem]\label{markaz}
 Let $G$ be a group. If  $[G:Z(G)]$ is finite then $G^{\prime}$ is finite.
  \end{theorem}

\begin{theorem}\label{es}
 The additive group $\mathbb{Q}$ of rational numbers  is a (unique) maximal
essential extension of group $\mathbb{Z}$ of integers as $\mathbb{Z}$-module.
  \end{theorem}

Now we can show the following:

\begin{theorem}
 Let $G$ be a locally center-by-finite group which is torsion-free. Then $P^*(G)$
is connected if and only if $G$ is isomorphic to a subgroup of $\mathbb{Q}$.
\end{theorem}

\begin{proof}
Let $P^*(G)$  be connected, and $x, y \in G$ nontrivial. There is a
path $x=a_1-a_2-\cdots -a_t=y$ in $P^*(G)$. Then
$\langle a_i\rangle \cap \langle a_{i+1}\rangle \neq \{1\}$ for $1\leq i \leq {t-1}$. This implies that $\langle a_i$ and $\langle a_{i+1}$ are
\emph{cmmensurable}, that is, their index has finite intersection in each.
Since commensurability is an equivalence relation, $\langle x\rangle$ and
$\langle y \rangle$ are commensurable, so their intersection is not trivial.

Firstly we prove that $G$ is abelian.
 Suppose to the contrary, we assume that $x, y$ are two elements with
 $xy \neq yx$. Set $H=\langle x, y \rangle$. By the hypothesis,
    $[H:Z(H)]$ is finite.    Now by Theorem \ref{markaz}, $H^{\prime}$ is
 a nontrivial finite group, which contradicts Lemma \ref{con}. We deduce that $G$ is abelian.

Let $a$ be a nontrivial element of $G$ and $H=\langle a \rangle$.
   By assumption $G$ is an essential  extension  of $H$ as $Z$-module. Let $L$ be a maximal
    essential  extension of $G$. By Theorem \ref{es}, $L\cong \mathbb{Q}$, which completes the proof.

For the converse, it is obvious that for any elements $x=m/n, y=m_1/n_1 \in\mathbb{Q}$, $nm_1x=n_1my$, as desired. \qed
    \end{proof}

There are examples of non-abelian torsion-free groups in \cite{ad}, \cite{obz}, in
    which the intersection of any two non-trivial subgroups is non-trivial. Thus their proper power graphs are connected.

For further investigation of the power graphs of torsion-free groups, we refer
to \cite{cgj}.

For   periodic groups, it seems  the following  result is the best  for connectivity.
\begin{theorem}[\cite{jj}, Lemma 2.1] \label{ham}
Let $G$ be a  periodic group. Then $P^*(G)$ is connected if and
only if for any two  elements $x,y$ of prime orders where $\langle
x \rangle \neq \langle y \rangle$, there exist elements
$x=x_0,x_1,\ldots,x_t=y$  such that $o(x_{2i})$ is prime,
$o(x_{2i+1})=o(x_{2i})o(x_{2i+2})$ for $i \in \{0,...,t/2\}$ and,
$x_i$ is adjacent to $x_{i+1}$ for $i \in \{0,1,...,t-1\}$.
\end{theorem}

By the above theorem, it can seen  that if $|\pi(Z(G))|\geq 2$, then $P^*(G)$ is connected, where $\pi(G)$ is the set of
all prime numbers $p$ such that $G$ has an element of order $p$.
(For suppose that $p,q\in\pi(Z(G))$. For any prime $r\in\pi(G)$, if $r\ne p$,
there is a path between any element of order $r$ and any central element of
order $p$; then it follows that all elements of prime order are in a single
component, so $P^*(G)$ is connected.)

Also the proper power graph of a $p$-group $G$ is connected if and only if
$G$ has exactly one subgroup of order $p$.
Moreover, if $G$ is a finite  $p$-group,  then $P^*(G)$ is connected if and
only if $G$ is cyclic or a generalized quaternion group $Q_{2^n}$.

We have the following result for the infinite case.

\begin{theorem}
Let $G$ be an infinite locally finite $p$-group. Then $P^*(G)$ is connected if and only if
$G\cong C_{p^{\infty}}$ for some prime number $p$, or
$G\cong Q_{2^{\infty}}$ where
$Q_{2^{\infty}}=\displaystyle{\bigcup_{i\geq3}Q_{2^i}}$.
\label{t:locfin}
\end{theorem}

\begin{proof}
By Theorem \ref{ham}, if $P^*(G)$ is connected, then $G$ has a unique subgroup
of order $p$. Hence $G$ is isomorphic to $Q_{2^{\infty}}$ or $C_{p^{\infty}}$.

Conversely, it is obvious that if $G$ is abelian then $G\cong C_{p^{\infty}}$.
Now we assume that $G$ is non-abelian. A theorem of Burnside (see
\cite[Theorem 12.5.2]{hall}) shows that a finite $2$-group with unique
subgroup of order~$2$ is generalized quaternion. Since $G$ is locally finite
it is a union of finite generalized quaternion groups, and so is 
$Q_{2^\infty}$. \qed
\end{proof}

Now we give a general class of examples.

\begin{example}
Let $G=\langle A,t\rangle$, where $A$ is an abelian torsion group of exponent
greater than $2$ and $t$ is an element of order $2$ inverting $A$. (This
includes dihedral groups of order greater than $4$.) Then $G$ contains
non-central involutions, for example $t$; these are isolated vertices in
$P^*(G)$. We note that the center of $G$ is a $2$-group.
\end{example}

Finally we return to the question of vertices joined to everything in infinite
groups. Our result is similar to the finite case.

\begin{theorem}
Let $G$ be an infinite group, and suppose that $x\in G$ has the property that
for any $y\in G$, either $y$ is a power of $x$ or \emph{vice versa}. Assume
that $x\ne1$. Then one of the following holds:
\begin{enumerate}\itemsep0pt
\item $G$ is infinite cyclic, and $x$ is a generator;
\item $G=C_{p^\infty}$ for some prime $p$, and $x$ is arbitrary;
\item $G=Q_{2^\infty}$ and $x$ has order~$2$.
\end{enumerate}
\end{theorem}

\begin{proof}
The hypothesis implies that $P^*(G)$ (obtained by deleting the identity from
$P(G)$) is connected; so by Lemma~\ref{con}, $G$ is either torsion-free or
a torsion group.

Suppose that $G$ is torsion-free; we claim that $G=\langle x\rangle$. If not,
take $y\notin\langle x\rangle$. Then $y$ is not a power of $x$, so $x$ is a
power of $y$. But the only elements in a cyclic group which are joined to all
elements are the identity and the generators; so
$\langle x\rangle=\langle y\rangle$, a contradiction.

Now suppose that $G$ is a torsion group. If $x$ does not have prime power order
then an element $y\notin\langle x\rangle$ would satisfy 
$x\in\langle y\rangle$, a contradiction as before. So we may assume that $x$
has prime power order. Every element $y\notin\langle x\rangle$ generates
a cyclic group containing $x$; so $G$ is locally a finite
$p$-group. Now Theorem~\ref{t:locfin} gives the result. \qed
\end{proof}

Now we can, as before, decide whether deleting all vertices which are joined
to everything leaves a connected graph:
\begin{enumerate}\itemsep0pt
\item If $G=\langle x\rangle$ is infinite cyclic, then $P(G)\setminus\{1,x^{\pm1}\}$ is connected with diameter~$2$, since $x^m$ and $x^n$ are both joined to
$x^{mn}$.
\item If $G=C_{p^\infty}$, then $P(G)$ is complete.
\item If $G=Q_{2^\infty}$ and $x$ has order~$2$, then $P(G)\setminus\{1,x\}$
consists of an infinite complete graph and infinitely many disjoint edges.
\end{enumerate}


\end{document}